\documentclass{cmslatex}
\usepackage{latexsym, amssymb, enumerate, amsmath}

\renewcommand{\theequation}{\arabic{section}.\arabic{equation}}

\newtheorem{thm}{Theorem}[section]

\newtheorem{example}[thm]{Example}

\newtheorem{rem}[thm]{Remark}

\begin{document}

\title{Sensitivity to switching rates in stochastically switched ODEs
\thanks{
}}
\author{Sean D. Lawley\thanks {Department of Mathematics, Duke University, Box 90320, Durham, NC 27708-0320 USA, (lawley@math.duke.edu).}%
\and Jonathan C. Mattingly \thanks {Department of Mathematics, Duke University, Box 90320, Durham, NC 27708-0320 USA, (jonm@math.duke.edu).}%
\and Michael C. Reed \thanks {Department of Mathematics, Duke University, Box 90320, Durham, NC 27708-0320 USA, (reed@math.duke.edu).}}




\pagestyle{myheadings} \markboth{STOCHASTICALLY SWITCHED ODES}{S. D. LAWLEY, J. C. MATTINGLY, AND M. C. REED}\maketitle

\begin{abstract}
  We consider a stochastic process driven by a linear ordinary
  differential equation whose right-hand side switches at exponential
  times between a collection of different matrices. We construct
  planar examples that switch between two matrices where
  the individual matrices and the average of the two matrices are all
  Hurwitz (all eigenvalues have strictly negative real part), but nonetheless the
  process goes to infinity at large time for certain values of the
  switching rate. We further construct examples in higher dimensions
  where again the two individual matrices and their averages are all
  Hurwitz, but the process has arbitrarily many transitions between
  going to zero and going to infinity at large time as the switching
  rate varies. In order to construct these examples, we first prove in
  general that if each of the individual matrices is Hurwitz, then the
  process goes to zero at large time for sufficiently slow switching
  rate and if the average matrix is Hurwitz, then the process goes to
  zero at large time for sufficiently fast switching rate. We also give simple conditions that ensure the process
  goes to zero at large time for all switching rates.
\end{abstract}


\section{Introduction}\label{intro}

We consider the stochastic process $(X_{t})_{t\ge0}\in\mathbb{R}^{d}$ where $X_{t}$ solves $\dot{X}_{t}=A_{I_{t}}X_{t}$ with $I_{t}$ a Markov process on a finite set $E$ and $\{A_{i}\}_{i\in E}$ a set of $d\times d$ real matrices. The stability of this system when the switching process $I_{t}$ is deterministic has been extensively studied in the past decade; see \cite{balde_note_2009} and \cite{lin09}.

In \cite{benaim12planar}, the authors study the stochastic problem in the plane with $I_{t}$ a Markov process and $E=\{0,1\}$. The authors assume both $A_{0}$ and $A_{1}$ are Hurwitz (all eigenvalues have strictly negative real part) and prove the surprising result that $||X_{t}||$ may converge to 0 or $+\infty$ as $t\to\infty$ depending on the switching rate as long as an average matrix $\bar{A} = \lambda A_{0}+(1-\lambda)A_{1}$ has a positive eigenvalue for some $\lambda\in(0,1)$.

In this paper, we show that the assumption that the average matrix has a positive eigenvalue is not necessary to ensure a blowup. Specifically, we construct examples in the plane where $A_{0}$, $A_{1}$, and $\bar{A} = \lambda A_{0}+(1-\lambda)A_{1}$ are all Hurwitz, but $||X_{t}||\to+\infty$ almost surely as $t\to\infty$ for certain values of the switching rate. This is significant for the general study of switching processes because it shows that the dynamics of the switching process can be very different from both the individual dynamics (in this case, the $A_{i}$'s) and the averaged dynamics (in this case, $\bar{A}$). These planar examples are also interesting because they have multiple transitions between $||X_{t}||$ going to $0$ and going to $+\infty$ at large time as the switching rate varies. Furthermore, we construct examples in higher dimensions that have arbitrarily many such phase transitions.

Recently researchers have devoted considerable attention to randomly
switched systems and we now comment on our work in this broader
context. \cite{hairer13}, \cite{benaim12quant}, \cite{benaim12qual},
and \cite{bakhtin12} all study invariant measures for such
processes. Our work shows that the existence of such invariant
measures may depend in a complicated way on the switching rates. In
\cite{hasler13finite}, \cite{hasler13asymptotic}, and \cite{belykh13},
the authors provide conditions under which their randomly switched
systems behave according to the individual systems for slow switching
and according to the averaged system for fast switching. We prove that
our system also obeys this principle in Theorems \ref{slow theorem}
and \ref{fast theorem}. However, we show in Example \ref{one
  transition} that the transition between the slow and fast switching
regimes can be quite complicated. Furthermore, Example \ref{many
  transitions} shows that it can be as complicated as we want.

As background for these surprising results, we first prove sufficient
conditions to ensure stability for all switching rates in Section
\ref{basic theorems}. Furthermore we also show in Section \ref{basic
  theorems} that the individual matrices determine the stability for
slow switching and that the average matrix determines the stability
for fast switching. In Section \ref{medium section} we use these
theorems to construct examples that show ``medium'' switching can
induce blowups even when the individual matrices and the average
matrix are all Hurwitz.

We conclude this introduction by defining notation. Let
$E=\{0,1,\dots,n-1\}$ and let $\{A_{i}\}_{i\in E}$ be a set of
$d\times d$ real matrices. For a given switching rate $r>0$, let
$(I_{t})_{t\ge0}$ be an irreducible continuous time Markov process
with state space $E$ and generator $rQ$. Under these assumptions, the
Markov process on $E$ with generator $rQ$ has a unique invariant
probability measure which we denote by $\pi$. Furthermore, $\pi$ is
the unique
probability vector satisfying $\pi Q = 0$.

Define $(X_{t})_{t\ge0}$ to be the solution of
\begin{align}\label{integral defn}
X_{t} = X_{0}+\int_{0}^{t}A_{I_{s}}X_{s}\,ds,\quad(t\ge0).
\end{align}
Then $(X_{t},I_{t})_{t\ge0}$ is a Markov process on $\mathbb{R}^{d}\times E$. Unless otherwise noted, assume throughout that the distribution of the initial condition $(X_{0},I_{0})$ is some given probability measure on $\mathbb{R}^{d}\times E$ satisfying $\mathbb{E}||X_{0}||<\infty$. Define the average matrix
\begin{align*}
\bar{A} = \sum_{i\in E}A_{i}\pi_{i}.
\end{align*}

The following description of our process will be useful. Let $\xi_{1},\xi_{2},\dots$ denote the succession of states visited by $I_{t}$, $\tau_{1},\tau_{2},\dots$ the holding times in each state, $N(t)$ the number of switches before $t$, and $a_{t} = t-\sum_{k=1}^{N(t)}\tau_{k}$ the time since the last switch. Observe that we can write $X_{t}$ as
\begin{align}\label{process defn}
X_{t} = \exp(A_{\xi_{N(t)+1}}a_{t})\exp(A_{\xi_{N(t)}}\tau_{N(t)})\dots \exp(A_{\xi_{1}}\tau_{1})X_{0}.
\end{align}

\section{Basic stability theorems}\label{basic theorems}

\begin{theorem}[normal case]\label{normal}
If $A_{i}$ is normal and Hurwitz for each $i\in E$, then $||X_{t}||\to0$ monotonically as $t\to\infty$ almost surely.
\end{theorem}

\begin{proof}
  Since each $A_{i}$ is normal and Hurwitz, there exists a $\gamma>0$
  so that for each $A_{i}$ and for every $t>0$,
\begin{align*}
||\exp(A_{i}t)||\le e^{-\gamma t}<1.
\end{align*}
Therefore
\begin{align*}
  ||X_{t}||
  & = ||\exp(A_{\xi_{N(t)+1}}a_{t})\exp(A_{\xi_{N(t)}}\tau_{N(t)})\dots\exp(A_{\xi_{1}}\tau_{1})X_{0}||\\
  & \le ||\exp(A_{\xi_{N(t)+1}}a_{t})||\Big(\prod_{k=1}^{N(t)}||\exp(A_{\xi_{k}}\tau_{k})||\Big)||X_{0}||\\
  & \le e^{-\gamma t}||X_{0}||\to0\quad\text{as }t\to\infty.
\end{align*}
To see that the convergence is monotonic, let $0\le s\le t$ and replace $X_{0}$ by $X_{s}$ in the calculation above.
\end{proof}

\begin{theorem}[commuting case]
 Assume $\{A_{i}\}_{i\in E}$ is a commuting family of matrices. If $\bar{A}$ is Hurwitz, then $||X_t||\to0$ as $t\to\infty$ almost surely.
\end{theorem}

\begin{proof}
Since $\bar{A}$ is Hurwitz, there exist positive $\beta$ and $\gamma$ so that for each $t\ge0$
\begin{align*}
||\exp(\bar{A}t)||\le\beta e^{-\gamma t}.
\end{align*}
For each $t>0$, define
\begin{align*}
C_{t} = \frac{1}{t}\Big(\sum_{k=1}^{N(t)}A_{\xi_{k}}\tau_k+A_{\xi_{N(t)+1}}a_t\Big)=\sum_{i\in E}A_{i}\frac{1}{t}\int_{0}^{t}1_{I_{s}=i}\,ds.
\end{align*}
Now since $\{A_{i}\}_{i\in E}$ is a commuting family of matrices, Equation (\ref{process defn}) becomes
\begin{align*}
||X_t||
 = ||\exp\Big(\sum_{k=1}^{N(t)}A_{\xi_{k}}\tau_k+A_{\xi_{N(t)+1}}a_t\Big)X_0|| = ||\exp\left(C_tt\right)X_0||\\
 = ||\exp\left(\bar{A}t\right)\exp\left((C_{t}-\bar{A})t\right)X_0||
 \le \beta e^{-\gamma t}e^{||C_{t}-\bar{A}||t}||X_0||,
\end{align*}
Since $Q$ is irreducible, $C_t\to\bar{A}$ almost surely as
$t\to\infty$ since $\frac1t \int_0^t 1_{I_s=i} ds \rightarrow \pi_i$
almost surely as $t \rightarrow \infty$ (see \cite{norris97}, page 126). Thus, $||X_t||\to0$ almost surely as $t\to\infty$.
\end{proof}

\begin{rem}
If $\{A_{i}\}_{i\in E}$ is a commuting family of matrices and each $A_{i}$ is Hurwitz, then $\bar{A}$ is Hurwitz. This is an immediate consequence of the fact that eigenvalues ``add'' - in some order - for commuting matrices.
\end{rem}


\begin{theorem}[slow switching]\label{slow theorem}
 Assume $A_{i}$ is Hurwitz for each $i\in E$. Then there exists a constant $a>0$ so that if $r<a$, then $||X_{t}||\to0$ as $t\to\infty$ almost surely.
\end{theorem}

\begin{proof}
Since each $A_{i}$ is Hurwitz, there exist $\beta>1$ and $\gamma>0$ so that for each $A_{i}$ and each $t\ge0$
\begin{align*}
||\exp(A_{i}t)||\le\beta e^{-\gamma t}.
\end{align*}
Therefore from Equation (\ref{process defn}), we have that
\begin{align}\label{commuting bound}
\begin{aligned}
||X_{t}||
& \le ||\exp(A_{\xi_{N(t)+1}}a_{t})||\Big(\prod_{k=1}^{N(t)}||\exp(A_{\xi_{k}}\tau_{k})||\Big)||X_{0}||\\
& \le \beta^{N(t)+1}e^{-\gamma t}||X_{0}||= \exp\bigg(\Big(\frac{N(t)+1}{t}\log\beta-\gamma\Big)t\bigg)||X_{0}||.
\end{aligned}
\end{align}
Next we claim that we have the following almost sure convergence as $K\to\infty$
\begin{align}\label{Birkhoff}
\frac{1}{K}\sum_{k=1}^{K}\tau_{k}\to(r\sum_{i\in E}\pi_{i}q_{i})^{-1},
\end{align}
where $q_{i}$ is the $i$th diagonal entry of $Q$. To see this, let $s_{j}^{i}$ denote the duration of the $j$th visit of the process $I_{t}$ to state $i\in E$ and let $V_{i}(K):=\sum_{k=1}^{K}1_{\xi_{k}=i}$ denote the number of visits to $i$ before the $K$th jump of the process $I_{t}$.
Then
\begin{align*}
\frac{1}{K}\sum_{k=1}^{K}\tau_{k} = \sum_{i\in E}\frac{1}{K}\sum_{j=1}^{V_{i}(K)}s_{j}^{i} = \sum_{i\in E}\frac{V_{i}(K)}{K}\frac{1}{V_{i}(K)}\sum_{j=1}^{V_{i}(K)}s_{j}^{i}.
\end{align*}
For each $i\in E$, $\frac{V_{i}(K)}{K}\to q_{i}\pi_{i}/(\sum_{k\in E}q_{k}\pi_{k})$ almost surely as $K\to\infty$ since $Q$ is irreducible. And by the strong law of large numbers, $\frac{1}{V_{i}(K)}\sum_{j=1}^{V_{i}(K)}s_{j}^{i}\to \frac{1}{rq_{i}}$ almost surely as $K\to\infty$. Therefore, Equation (\ref{Birkhoff}) is verified.

By the definition of $N(t)$ we have that $\sum_{k=1}^{N(t)}\tau_{k}\le t\le\sum_{k=1}^{N(t)+1}\tau_{k}$. Therefore
\begin{align}\label{squeeze}
\frac{\sum_{k=1}^{N(t)}\tau_{k}}{N(t)}\le\frac{t}{N(t)}\le\frac{\sum_{k=1}^{N(t)+1}\tau_{k}}{N(t)+1}\frac{N(t)+1}{N(t)}.
\end{align}
Since each $\tau_{k}$ is almost surely finite, $N(t)\to\infty$ almost surely as $t\to\infty$. It then follows from combining Equations (\ref{Birkhoff}) and (\ref{squeeze}) that
\begin{align*}
\frac{N(t)}{t}\to r\sum_{i\in E}\pi_{i}q_{i}\quad\text{almost surely as }t\to\infty.
\end{align*}
So if $r<\gamma(2\log\beta\sum_{i\in E}\pi_{i}q_{i})^{-1}$, then $||X_{t}||\to0$ almost surely as $t\to\infty$ by Equation (\ref{commuting bound}).
\end{proof}


\begin{theorem}[fast switching]\label{fast theorem}
 Assume $\bar{A}$ is Hurwitz. Then there exists a constant $b>0$ so that if $r>b$, then $||X_{t}||\to0$ as $t\to\infty$ almost surely.
\end{theorem}

The proof relies on the following lemma. Let $\mathbb{E}_{\nu}$ denote the
expectation with respect to the measure of the process
$(I_{t})_{t\ge0}$ with $I_{0}$ distributed according to $\nu$. Since we will consider processes with different switching rates, let us momentarily make the dependence on the switching rate explicit by letting $(I^{(r)}_{t})_{t\ge0}$ be the Markov process on $E$ with generator $rQ$ and define $(X^{(r)}_{t})_{t\ge0}$ with respect to $(I^{(r)}_{t})_{t\ge0}$ as before. Define $S^{(r)}_{t}$ to be the operator that
maps $X_{0}$ to $X^{(r)}_{t}$. Observe that $S^{(r)}_{t}$ is a function of
$(I^{(r)}_{s})_{0\le s\le t}$.

\begin{lemma}\label{lemma Kurtz}
For every probability measure $\nu$ on $E$ and for every $t>0$,
\begin{align*}
\mathbb{E}_{\nu}||S^{(r)}_{t}||\to ||\exp(\bar{A}t)||\quad\text{as }r\to\infty.
\end{align*}
\end{lemma}
\begin{proof}
Define $\{\xi^{1}_{i}\}_{i=1}^{\infty}$,
  $\{\tau^{1}_{i}\}_{i=1}^{\infty}$, $\{N^{1}(t)\}_{t\ge0}$, and
  $\{a^{1}_{t}\}_{t\ge0}$ as before but now with respect to
  $\{I^{(1)}_{t}\}_{t\ge0}$. Let the
  distribution of $I_{0}$ be a given probability measure $\nu$ on $E$
  and for $\lambda>0$ define
\begin{align*}
\tilde{S}^{\lambda}_{t}
& = \exp\Big(A^{T}_{\xi^{1}_{1}}\frac{\tau^{1}_{1}}{\lambda}\Big)\exp\Big(A^{T}_{\xi^{1}_{2}}\frac{\tau^{1}_{2}}{\lambda}\Big)\dots\exp\Big(A^{T}_{\xi_{N^{1}(\lambda t)}}\frac{\tau_{N^{1}(\lambda t)}}{\lambda}\Big)\exp\Big(A^{T}_{\xi_{N^{1}(\lambda t)+1}}\frac{a^{1}_{\lambda t}}{\lambda}\Big)
\end{align*}
where we denote the transpose of a matrix $B$ by $B^{T}$. Observe that if $r=\lambda$, then $\tilde{S}^{\lambda}_{t}$ has been defined so that $(\tilde{S}^{\lambda}_{t})^{T}$ and $S^{(r)}_{t}$ are equal in distribution. 

By \cite{kurtz72}, $\tilde{S}^{\lambda}_{t}\to \exp(\bar{A}^{T}t)$ almost surely in the strong operator topology as $\lambda\to\infty$. Since $\mathbb{R}^{d}$ is finite-dimensional, we actually have that the convergence holds in the uniform operator topology. Since $||B||=||B^{T}||$ for every matrix $B$, it follows that 
\begin{align*}
||\tilde{S}^{\lambda}_{t}||\to ||\exp(\bar{A}t)||\text{ almost surely as }\lambda\to\infty.
\end{align*}
Since $||\tilde{S}^{\lambda}_{t}||\le\exp(\max_{i}||A_{i}||t)$ for every $\lambda>0$, the bounded convergence theorem gives
\begin{align*}
\mathbb{E}||\tilde{S}^{\lambda}_{t}||\to ||\exp(\bar{A}t)||\text{ as }\lambda\to\infty.
\end{align*}
Since $||\tilde{S}^{\lambda}_{t}||$ and $||S^{r}_{t}||$ are equal in distribution, the proof is complete.
\end{proof}

\begin{proof}[of Theorem \ref{fast theorem}]
Since $\bar{A}$ is Hurwitz, there exist positive numbers $\beta$ and $\gamma$ so that for every $t\ge0$
\begin{align*}
||\exp(\bar{A}t)||\le\beta e^{-\gamma t}.
\end{align*}
Thus we can choose $T>0$ so that $||\exp(\bar{A}T)||<\frac{1}{4}$. By Lemma \ref{lemma Kurtz} there exists a $b>0$ so that if $r>b$, then $\mathbb{E}_{i}||S^{(r)}_{T}||<\frac{1}{2}$ for each $i\in E$, where $\mathbb{E}_{i}$ denotes the expectation with respect to the measure of the process $(I_{t})_{t\ge0}$ with initial measure $\mathbb{P}(I_{0}=i)=1$.

Let $r>b$ and define the process $\{M_{n}\}_{n=0}^{\infty}$ and the filtration $\{\mathcal{F}_{n}\}_{n=0}^{\infty}$ by
\begin{align*}
\begin{aligned}[c]
M_{n} & =||X_{nT}||
\end{aligned}
\quad\text{and }\quad
\begin{aligned}[c]
\mathcal{F}_{n} & =\sigma((X_{t},I_{t}):0\le t \le n T).
\end{aligned}
\end{align*}
We claim that $M_{n}$ is a supermartingale with respect to $\mathcal{F}_{n}$. It's immediate that $M_{n}\in\mathcal{F}_{n}$ and $\mathbb{E}|M_{n}|\le e^{\Lambda nT}<\infty$ for $\Lambda:=\max_{i\in E}||A_{i}||$. For $0\le s\le t$, define $S(s,t)$ to be the operator that maps $X_{s}$ to $X_{t}$. We now check the supermartingale property.
\begin{align*}
\mathbb{E}[M_{n+1} |\mathcal{F}_{n}]
& \le \mathbb{E}[||S(nT,(n+1)T)|| \,||X_{nT}|| |\mathcal{F}_{n}]\\
& = M_{n}\mathbb{E}[||S(nT,(n+1)T)|| |\mathcal{F}_{n}]\\
& = M_{n}\mathbb{E}_{I_{nT}}||S(nT,(n+1)T)||\\
& \le \frac{1}{2}M_{n}.
\end{align*}
Taking the expectation of the above inequality and iterating yields $\mathbb{E}M_{n}\le\frac{1}{2^{n}}\mathbb{E}M_{0}$. Therefore $M_{n}$ converges in $L^{1}$ to 0 since $M_{n}\ge0$. Also since $M_{n}\ge0$, the martingale convergence theorem implies that $M_{n}$ must converge almost surely. Therefore $M_{n}$ converges almost surely to 0.

To conclude that $||X_{t}||\to0$ almost surely, we need to control $||X_{t}||$ at times between multiples of $T$. This is easily obtained since $||X_{t}||$ cannot grow faster than $e^{\Lambda t}$. Let $\omega\in\Omega$ be such that $M_{n}(\omega)\to0$ and let $\epsilon>0$. There exists $N=N(\omega,\epsilon)$ so that for all $n\ge N$,
\begin{align*}
||M_{n}(\omega)||< e^{-\Lambda T}\epsilon.
\end{align*}
Thus for all $t\ge NT$,
\begin{align*}
||X_{t}(\omega)||
 \le ||(S(t-T\lfloor t/T\rfloor,t)X_{T\lfloor t/T\rfloor})(\omega)||
 \le e^{\Lambda T}M_{\lfloor t/T\rfloor}(\omega)< \epsilon.
\end{align*}
Since this set of $\omega$'s has probability one, the proof is complete.
\end{proof}

\begin{example}
Assume $E=\{0,1\}$ and $Q = \left( \begin{smallmatrix} -1&1\\ 1&-1 \end{smallmatrix} \right)$.
Define
\begin{align*}
\begin{aligned}[c]
A_{0} & = \begin{pmatrix} 1&4\\ 0&-2 \end{pmatrix}
\end{aligned}
\qquad
\begin{aligned}[c]
A_{1} & = \begin{pmatrix} -2&0\\ 0&1 \end{pmatrix}.
\end{aligned}
\end{align*}
Then $A_{0}$ and $A_{1}$ each have a positive eigenvalue, but $\bar{A} = \frac{1}{2}(A_{0}+A_{1})$ is Hurwitz. So despite the fact that each individual matrix is unstable, Theorem \ref{fast theorem} guarantees that $||X_{t}||\to0$ almost surely as $t\to\infty$ for sufficiently fast switching rate.
\end{example}

\section{Medium switching can be complicated}\label{medium section} We will now construct a switching example with two matrices, $A_{0}$ and $A_{1}$, that is surprising for the following two reasons. First, the individual matrices $A_{0}$ and $A_{1}$ and the average $\bar{A}=\frac{1}{2}(A_{0}+A_{1})$ are all Hurwitz, but $||X_{t}||$ will still blow up at large time for certain values of the switching rate. In \cite{benaim12planar}, the authors show that $||X_{t}||$ can blow up if the two individual matrices $A_{0}$ and $A_{1}$ are Hurwitz as long as the average matrix has a positive eigenvalue. Thus our result shows that this assumption on the average matrix is not necessary.

Second, the asymptotic behavior of the following example has multiple ``phase transitions'' as the switching rate varies. That is, the process goes to zero at large time for both slow and fast switching, but blows up for medium switching.

We also remark that we can choose the negative real part of all the eigenvalues of $A_{0}$, $A_{1}$, and $\bar{A}$ to have arbitrarily large absolute value.

\begin{example}\label{one transition}
Assume $\mathbb{P}(X_{0}=0)=0$ and let $E=\{0,1\}$ and $Q = \left( \begin{smallmatrix} -1&1\\ 1&-1 \end{smallmatrix} \right)$. We will show the existence of matrices $A_{0},A_{1}\in\mathbb{R}^{2\times2}$ and positive numbers $a<b$, so that
\begin{enumerate}
\item
$A_{0}$, $A_{1}$ are each Hurwitz.
\item
$\bar{A}=\frac{1}{2}(A_{0}+A_{1})$ is Hurwitz.
\item
If $r\notin(a,b)$, then $||X_{t}||\to0$ almost surely as $t\to\infty$.
\item
$||X_{t}||\to\infty$ almost surely as $t\to\infty$ for some value of $r\in(a,b)$.
\end{enumerate}
\end{example}

For positive $\alpha$ and $c$, we define
\begin{align*}
\begin{aligned}[c]
A_{0} & = \begin{pmatrix} -\alpha&c\\ 0&-\alpha \end{pmatrix}
\end{aligned}
\qquad
\begin{aligned}[c]
A_{1} & = \begin{pmatrix} -\alpha&0\\ -c&-\alpha \end{pmatrix}.
\end{aligned}
\end{align*}
Observe that $A_{0}$ and $A_{1}$ each have $-\alpha<0$ as their only
eigenvalue. The two eigenvalues of $\bar{A}=\frac{1}{2}(A_{0}+A_{1})$
are $-\alpha\pm ic/2$. Thus $A_{0}$, $A_{1}$, and $\bar{A}$ are each
Hurwitz. By Theorems \ref{slow theorem} and \ref{fast theorem},
$||X_{t}||\to0$ as $t\to\infty$ almost surely for sufficiently large
$r$ and for sufficiently small $r$. We will show that
$||X_{t}||\to+\infty$ as $t\to\infty$ almost surely for some
intermediate values of $r$.

We use polar coordinates to study the large time behavior of
$||X_{t}||$. Our technique follows \cite{benaim12planar} in this
setting and the well known utility of  the polar representation when
studying Lyapunov exponents (especially  in
two-dimensions) which dates back to at least \cite{Khasminskii_1967}. Define the radial process $R_{t}:=||X_{t}||$ and
define the angular process $U_{t}$ as the point on the unit circle
$S^{1}$ given by $X_{t}/R_{t}$. A short calculation shows that between
jumps $R_{t}$ and $U_{t}$ satisfy
\begin{align}
\dot{R}_{t} & = R_{t}\langle A_{I_{t}}U_{t},U_{t}\rangle\label{radius ode}\\
\dot{U}_{t} & = A_{I_{t}}U_{t} - \langle A_{I_{t}}U_{t},U_{t}\rangle U_{t}\label{angular ode}.
\end{align}

The advantage of this decomposition is that the evolution of the angular process doesn't depend on the radial process. Therefore $(U_{t},I_{t})$ is a Markov process on $S^{1}\times \{0,1\}$.

\begin{lemma}\label{invariant measure}
If we identify $\theta\in\mathbb{R}$ with $(\cos\theta,\sin\theta)\in S^{1}$, then the unique invariant measure of the angular process $U_{t}$ is given by
\begin{align*}
\mu(d\theta,i) & = p_{i}(\theta;r/c)1_{[0,2\pi]}(\theta)\,d\theta
\end{align*}
where for any parameter $\lambda>0$, the functions $p_{0}$ and $p_{1}$ satisfy
\begin{align}
p_{i}(\theta;\lambda) &= p_{1-i}(\theta+\pi/2;\lambda)=p_{i}(\theta+\pi;\lambda)\quad\text{for }\theta\in\mathbb{R}\label{periodic},\\
\text{and}\quad p_{0}(\theta;\lambda)&<p_{1}(\theta;\lambda)\quad\text{for }\theta\in(-\frac{\pi}{2},0)\label{greater than}.
\end{align}
\end{lemma}
\begin{proof}
Define the process $\Theta_{t}\in\mathbb{R}$ to be the lift of $U_t
\in S^1$
from the circle to its covering space $\mathbb{R}$. That is to say
$\Theta_t$ is the unique process 
so that $U_{t}=(\cos\Theta_{t},\sin\Theta_{t})$, $\Theta_{t}$ is continuous in $t$, and $\Theta_{0}\in[0,2\pi)$. It follows from Equation (\ref{angular ode}) and plugging in our values for $A_{0}$ and $A_{1}$ that between jumps $\Theta_{t}$ satisfies
\begin{align*}
\dot{\Theta}_{t} & =
-c[I_{t}\cos^{2}(\Theta_{t})+(1-I_{t})\sin^{2}(\Theta_{t})]\le0.
\end{align*}
Since $\min_{i\in\{0,1\}}-c[i\cos^{2}(\theta)+(1-i)\sin^{2}(\theta)]\le-c/2<0$ for all $\theta\in\mathbb{R}$, it follows that $\Theta_{t}\to-\infty$ as $t\to\infty$ almost surely. Since $\Theta_{t}$ is continuous, we conclude that the Markov process $(U_{t},I_{t})$ is recurrent and irreducible and must have a unique invariant measure.

If we identify $\theta\in\mathbb{R}$ with $(\cos\theta,\sin\theta)\in S^{1}$, then the adjoint of generator of the Markov process $(U_{t},I_{t})$ is
\begin{align*}
(\mathcal{L}^{*}q)(\theta,i) & = \partial_{\theta}\left(c\left[(1-i)\sin^{2}(\theta)+i\cos^{2}(\theta)\right]q(\theta,i)\right) + r(q(\theta,1-i)-q(\theta,i)).
\end{align*}
For $\theta\in(-\frac{\pi}{2},0)$ and $\lambda>0$, define
\begin{align*}
H(\theta;\lambda)
 & = \exp{\left(-2\lambda\cot(2\theta)\right)}\int_{\theta}^{0}\exp{\left(2\lambda\cot(2y)\right)}\sec^{2}(y)\,dy\\
p_{0}(\theta;\lambda) & = C\csc^{2}(\theta)\lambda H(\theta)\\
p_{1}(\theta;\lambda) & = C\sec^{2}(\theta)\left[1-\lambda H(\theta)\right].
\end{align*}
where 
\begin{align*}
C(\lambda) = \left[4\int_{-\frac{\pi}{2}}^{0}\sec^{2}(x)+(\csc^{2}(x)-\sec^{2}(x))\lambda H(x)\,dx\right]^{-1}.
\end{align*}
Define $H(0;\lambda)=0=p_{0}(0;\lambda)$ and $p_{1}(0;\lambda)=C(\lambda)$. Extend $p_{1}$ and $p_{0}$ to be defined on the rest of the real line by Equation (\ref{periodic}). It is easy to check that these three functions are well-defined.

Writing $p_{i}(\theta;\lambda)$ as $p(\theta,i;\lambda)$, it is easy to check that $\mathcal{L}^{*}p(\theta,i;\lambda)=0$ for all $\theta\in\mathbb{R}$ and for $i=\{0,1\}$. Thus, the measure $\mu$ defined in the statement of the lemma is the unique invariant measure for $(U_{t},I_{t})$.

We now check that $p_{0}$ and $p_{1}$ satisfy Equation (\ref{greater than}). Let $\lambda>0$ and observe that for $\theta\in(-\frac{\pi}{2},0)$, writing $1=\sin^{2}(y)\csc^{2}(y)$ in the integrand gives
\begin{align}\label{H less than sin}
\begin{aligned}
H(\theta;\lambda)
 & = \exp{\left(-2\lambda\cot(2\theta)\right)}\int_{\theta}^{0}\exp{\left(2\lambda\cot(2y)\right)}\sec^{2}(y)\sin^{2}(y)\csc^{2}(y)\,dy\\
 & < \exp{\left(-2\lambda\cot(2\theta)\right)}\sin^{2}(\theta)\int_{\theta}^{0}\exp{\left(2\lambda\cot(2y)\right)}\sec^{2}(y)\csc^{2}(y)\,dy\\
 & = \frac1\lambda\sin^{2}(\theta),
 \end{aligned}
\end{align}
since $\sin^{2}(\theta)$ is strictly decreasing on $(-\frac{\pi}{2},0)$ and
\begin{align*}
\frac{d}{dy}\left[\exp{\left(2\lambda\cot(2y)\right)}\right]
=-\lambda\exp{\left(2\lambda\cot(2y)\right)}\sec^{2}(y)\csc^{2}(y).
\end{align*}
Observe also that for $\theta\in(-\frac{\pi}{2},0)$
\begin{align}\label{H prime}
H'(\theta;\lambda) = \lambda H(\theta;\lambda)(\sec^{2}(\theta)+\csc^{2}(\theta))-\sec^{2}(\theta) = \frac{1}{C}(p_{0}(\theta;\lambda)-p_{1}(\theta;\lambda)).
\end{align}
Combining Equations (\ref{H less than sin}) and (\ref{H prime}), we have that for $\theta\in(-\frac{\pi}{2},0)$
\begin{align*}
\frac{1}{C}(p_{0}(\theta;\lambda)-p_{1}(\theta;\lambda))
< 0.
\end{align*}
Thus Equation (\ref{greater than}) holds.
\end{proof}

\begin{lemma}\label{G}
For $\lambda>0$, define
\begin{align*}
G(\lambda)& :=\int_{0}^{2\pi} (p_{0}(\theta;\lambda)-p_{1}(\theta;\lambda))\cos(\theta)\sin(\theta)\,d\theta.
\end{align*}
Then $G(\lambda)>0$ and 
\begin{itemize}
\item\label{fa}
If $G\left(\frac{r}{c}\right) > \frac{\alpha}{c}$, then $||X_{t}||\to\infty$ as $t\to\infty$ almost surely.
\item
If $G\left(\frac{r}{c}\right) < \frac{\alpha}{c}$, then $||X_{t}||\to0$ as $t\to\infty$ almost surely.
\end{itemize}
\end{lemma}
\begin{proof}
By Equations (\ref{periodic}) and (\ref{greater than}) in the statement of Lemma \ref{invariant measure}, we have that $(p_{0}(\theta;\lambda)-p_{1}(\theta;\lambda))\cos(\theta)\sin(\theta)>0$ for all $\theta$ and thus $G(\lambda)>0$.

Now by Equation (\ref{radius ode}), we have that
\begin{align*}
\frac{1}{t}\log\left(\frac{R_{t}}{R_{0}}\right) = \frac{1}{t}\int_{0}^{t}\langle A_{I_{s}}U_{s},U_{s}\rangle\,ds.
\end{align*}
Identify $\theta\in\mathbb{R}$ with $e_{\theta}:=(\cos\theta,\sin\theta)\in S^{1}$. It follows from Lemma \ref{invariant measure} and Birkhoff's ergodic theorem that there exists a set $A\in S^{1}$ with $\mu(A)=1$ so that if $U_{0}\in A$, then
\begin{align}\label{angle convergence}
\frac{1}{t}\log\left(\frac{R_{t}}{R_{0}}\right) &\to \int\langle A_{i}e_{\theta},e_{\theta}\rangle\,\mu(d\theta,i)
\quad\text{almost surely as }t\to\infty.
\end{align}
Define $T_{A}:=\inf\{t\ge0:U_{t}\in A\}$ and observe that for any $U_{0}\in S^{1}$, we have that $T_{A}<\infty$ almost surely since $U_{t}$ is recurrent. Since $T_{A}$ is a stopping time, we have that the convergence in Equation (\ref{angle convergence}) actually holds for every $U_{0}\in S^{1}$.

Plugging in our choice of $A_{0}$ and $A_{1}$ and the definition of $\mu$ yields
\begin{align*}
\int\langle A_{i}e_{\theta},e_{\theta}\rangle\,\mu(d\theta,i)
& = \int_{0}^{2\pi}\langle A_{0}e_{\theta},e_{\theta}\rangle p_{0}(\theta;r/c)\,d\theta
+ \int_{0}^{2\pi}\langle A_{1}e_{\theta},e_{\theta}\rangle p_{1}(\theta;r/c)\,d\theta\\
& = c\int_{0}^{2\pi}(p_{0}(\theta;r/c)-p_{1}(\theta;r/c))\cos(\theta)\sin(\theta)\,d\theta
- \alpha\\
& = cG\left(\frac{r}{c}\right) - \alpha.
\end{align*}
Hence if $G\left(\frac{r}{c}\right) > \frac{\alpha}{c}$, then $\lim_{t\to\infty}\frac{1}{t}\log\left(\frac{R_{t}}{R_{0}}\right)>0$ almost surely and thus $||X_{t}||\to\infty$ as $t\to\infty$ almost surely. Similarly if $G\left(\frac{r}{c}\right) < \frac{\alpha}{c}$, then $||X_{t}||\to0$ as $t\to\infty$ almost surely.
\end{proof}

Since $G\left(\frac{r}{c}\right)>0$ for every pair of positive numbers $r$ and $c$, it is immediate that we can choose $r$, $c$, and $\alpha$ so that $||X_{t}||\to\infty$ as $t\to\infty$ almost surely.

\begin{rem}
Relating this example to the deterministic problem studied in \cite{balde_note_2009}, the pair $A_{0}$, $A_{1}$ defined above fall in to case \textbf{S4} with $\mathcal{R}>1$ of Theorem 1 in \cite{balde_note_2009}.
\end{rem}

\subsection{Many  transitions between stable and unstable}

The following example shows that there exist two matrices such that as
the switching rate varies from zero to infinity, the asymptotic
behavior of the system will switch between converging to zero and
converging to infinity at least any prespecified number of times.

\begin{example}\label{many transitions}
  Assume $\mathbb{P}(X_{0}=0)=0$ and let $E=\{0,1\}$ and $Q =
  \left( \begin{smallmatrix} -1&1\\ 1&-1 \end{smallmatrix}
  \right)$. We will show that for any positive integer $k$, there
  exist matrices $A_{0},A_{1}\in\mathbb{R}^{2k\times2k}$ and positive
  numbers $a_{1}<b_{1}<a_{2}<b_{2}<\dots<a_{k}<b_{k}$ so that
\begin{enumerate}
\item
$A_{0}$, $A_{1}$ are each Hurwitz.
\item
$\bar{A}=\frac{1}{2}(A_{0}+A_{1})$ is Hurwitz.
\item
If $r\notin\bigcup_{i=1}^{k}(a_{i},b_{i})$, then $||X_{t}||\to0$ almost surely as $t\to\infty$.
\item
For every $i\in\{1,\dots,k\}$, $||X_{t}||\to\infty$ almost surely as $t\to\infty$ for some value of $r\in(a_{i},b_{i})$.
\end{enumerate}
\end{example}

Let $k$ be a given positive integer and define the two block diagonal matrices $A_{0},A_{1}\in\mathbb{R}^{2k\times2k}$ by
\begin{align}\label{blocks}
\begin{aligned}[c]
A_{0} & = \begin{pmatrix} 
A^{1}_{0} & 0 & \cdots & 0 \\ 0 & A^{2}_{0} & \cdots &  0 \\
\vdots & \vdots & \ddots & \vdots \\
0 & 0 & \cdots & A^{k}_{0} 
\end{pmatrix}
\end{aligned}
\qquad
\begin{aligned}[c]
A_{1} & = \begin{pmatrix} 
A^{1}_{1} & 0 & \cdots & 0 \\ 0 & A^{2}_{1} & \cdots &  0 \\
\vdots & \vdots & \ddots & \vdots \\
0 & 0 & \cdots & A^{k}_{1}
\end{pmatrix}
\end{aligned}
\end{align}
where
\begin{align}\label{singles}
\begin{aligned}[c]
A^{i}_{0} & = \begin{pmatrix} -\alpha_{i}&c_{i}\\ 0&-\alpha_{i} \end{pmatrix}
\end{aligned}
\qquad
\begin{aligned}[c]
A^{i}_{1} & = \begin{pmatrix} -\alpha_{i}&0\\ -c_{i}&-\alpha_{i}\end{pmatrix}
\end{aligned}
\end{align}
for some positive numbers $\{c_{i}\}_{i=1}^{k}$ and
$\{\alpha_{i}\}_{i=1}^{k}$. It's immediate that $A_{0}$, $A_{1}$, and
$\bar{A}$ are all Hurwitz.

Let $X_{t}$ denote the $\mathbb{R}^{2k}$-valued process corresponding
to (\ref{blocks}) and $X^{(i)}_{t}$ the $\mathbb{R}^{2}$-valued process
corresponding to (\ref{singles}) for each $i\in\{1,\dots,k\}$. Since
the ODEs for $X^{(i)}$ and $X^{(j)}$ are not coupled for $i\ne j$, we have
that $X_t=(X_t^{1)}, \dots,X_t^{(k)})$ when viewed as an
$(\mathbb{R}^{2})^k$-valued process. In particular, one has
\begin{align*}
||X_{t}||^{2}=\sum_{i=1}^{k}||X^{(i)}_{t}||^{2}.
\end{align*}
Thus $||X_{t}||\to0$ if and only if $||X^{(i)}_{t}||\to0$ for every
$i\in\{1,\dots,k\}$. Furthermore if $||X^{(i)}_{t}||\to\infty$ for some
$i\in\{1,\dots,k\}$, then $||X_{t}||\to\infty$.

The proof proceeds by choosing the parameters $\alpha_i$ and $c_i$ as
in Example~\ref{one transition} so that $X^{(i)}$ is unstable for
switching rates $r$ in an interval $(a_i,b_i)$ but stable out side of
the interval. By arranging so that the collection of intervals
$\{(a_j,b_j) : j =1,\dots,k\} $ are disjoint we will  succeed at
constructing the desired matrices $A_0$ and $A_1$.

More explicitly, it follows from Lemma~\ref{G} and Theorems~\ref{slow
  theorem} and \ref{fast theorem} that we can choose $r_{1}$, $c_{1}$,
$\alpha_{1}$, and $a_{1}<b_{1}$ so that
\begin{align*}
\begin{aligned}[c]
G\left(\frac{r_{1}}{c_{1}}\right)>\frac{\alpha_{1}}{c_{1}}
\end{aligned}
\quad\text{and}\quad
\begin{aligned}[c]
  G\left(\frac{r}{c_{1}}\right)<\frac{\alpha_{1}}{c_{1}}\text{ if
  }r\notin(a_{1},b_{1}).
\end{aligned}\end{align*}
Choose $N>\frac{b_{1}}{a_{1}}$ and for $i\in\{2,\dots,k\}$ define
\begin{align*}
\begin{aligned}[c]
a_{i}=\frac{a_{1}}{N^{i-1}},
\end{aligned}
\quad
\begin{aligned}[c]
b_{i}=\frac{b_{1}}{N^{i-1}},
\end{aligned}
\quad
\begin{aligned}[c]
\alpha_{i}=\frac{\alpha_{1}}{N^{i-1}},
\end{aligned}
\quad
\begin{aligned}[c]
c_{i}=\frac{c_{1}}{N^{i-1}},
\end{aligned}
\quad
\begin{aligned}[c]
r_{i}=\frac{r_{1}}{N^{i-1}}.
\end{aligned}
\end{align*}
To see that our intervals $(a_{i},b_{i})$ don't overlap, observe that
$a_{i}<b_{i}$ for each $i$ and
\begin{align*}
  b_{i}=\frac{b_{1}}{N^{i-1}} < \frac{Na_{1}}{N^{i-1}} = a_{i-1}.
\end{align*}
Next observe that if $r\notin(a_{i},b_{i})$, then
$rN^{i-1}\notin(a_{1},b_{1})$ and therefore
\begin{align*}
G\left(\frac{r}{c_{i}}\right) = G\left(\frac{rN^{i-1}}{c_{1}}\right)<\frac{\alpha_{1}}{c_{1}}.
\end{align*}
Thus, $||X^{(i)}_{t}||\to0$ almost surely as $t\to\infty$ if $r\notin(a_{i},b_{i})$.

Finally observe that $r_{i}\in(a_{i},b_{i})$ and
\begin{align*}
  G\left(\frac{r_{i}}{c_{i}}\right) =
  G\left(\frac{r_{1}}{c_{1}}\right)>\frac{\alpha_{1}}{c_{1}} =
  \frac{\alpha_{i}}{c_{i}}.
\end{align*}
Thus, $||X^{(i)}_{t}||\to\infty$ almost surely as $t\to\infty$ if the
switching rate is $r_{i}\in(a_{i},b_{i})$.

\section{Conclusions}
Stochastically switched linear ODEs are one of the simplest examples
of stochastically switched systems. However despite their simplicity,
we have shown that their behavior can be quite rich. First, the large
time behavior can depend on the switching rate in a very delicate
way. Second, this large time behavior can be very different from the
large time behavior of both the individual systems and the average
system.

\medskip

{\bf Acknowledgement.}  JCM would like to thank Yuri Bakhtin for
stimulating discussions. This research was partially supported by NSF
grants EF-1038593 (HFN, MCR), DMS-0854879 (JCM), DMS-0943760 (MCR),
and NIH grant R01 ES019876 (DT).  \medskip

\bibliography{switchingbib}

\begin{thebibliography}{10}

\bibitem{bakhtin12}
{\sc Y.~Bakhtin and T.~Hurth}, {\em Invariant densities for dynamical systems
  with random switching}, Nonlinearity, 25 (2012).

\bibitem{balde_note_2009}
{\sc M.~Balde, U.~Boscain, and P.~Mason}, {\em A note on stability conditions
  for planar switched systems}, International Journal of Control, 82 (2009),
  pp.~1882--1888.

\bibitem{belykh13}
{\sc I.~Belykh, V.~Belykh, R.~Jeter, and M.~Hasler}, {\em Multistable randomly
  switching oscillators: the odds of meeting a ghost}, European Physical
  Journal Special Topics,  (2013).

\bibitem{benaim12qual}
{\sc M.~Benaim, S.~Leborgne, F.~Malrieu, and P.-A. Zitt}, {\em Qualitative
  properties of certain piecewise deterministic markov processes}, preprint,
  (2012).

\bibitem{benaim12quant}
\leavevmode\vrule height 2pt depth -1.6pt width 23pt, {\em Quantitative
  ergodicity for some switched dynamical systems}, Electronic Communications in
  Probability, 17 (2012), pp.~1--14.

\bibitem{benaim12planar}
\leavevmode\vrule height 2pt depth -1.6pt width 23pt, {\em On the stability of
  planar randomly switched systems}, Annals of Applied Probability,  (2013).

\bibitem{hairer13}
{\sc B.~Cloez and M.~Hairer}, {\em Exponential ergodicity for markov processes
  with random switching}, preprint,  (2013).

\bibitem{hasler13finite}
{\sc M.~Hasler, V.~Belykh, and I.~Belykh}, {\em Dynamics of stochastically
  blinking systems. part i: Finite time properties}, SIAM J. Applied Dynamical
  Systems, 12 (2013), pp.~1007--1030.

\bibitem{hasler13asymptotic}
\leavevmode\vrule height 2pt depth -1.6pt width 23pt, {\em Dynamics of
  stochastically blinking systems. part ii: Asymptotic properties}, SIAM J.
  Applied Dynamical Systems, 12 (2013), pp.~1031--1084.

\bibitem{Khasminskii_1967}
{\sc R.~Z. Khas’minskii}, {\em Necessary and sufficient conditions for the
  asymptotic stability of linear stochastic systems}, Theory of Probability and
  Its Applications, 12 (1967), pp.~144--147.

\bibitem{kurtz72}
{\sc T.~Kurtz}, {\em A random trotter product formula}, Proceedings of the
  American Mathematical Society, 35 (1972).

\bibitem{lin09}
{\sc H.~Lin and P.~J. Antsaklis}, {\em Stability and stabilizability of
  switched linear systems: A survey of recent results}, IEEE Transaction on
  Automatic Control, 54 (2009).

\bibitem{norris97}
{\sc J.~Norris}, {\em Markov Chains}, Cambridge University Press, 1997.

\end{thebibliography}
\bibliographystyle{siam}


          %


          %



\end{document}